\documentclass[10pt]{amsart}
\usepackage{amsfonts}
\usepackage{ifthen}
\usepackage{amsthm}
\usepackage{amsmath,mathrsfs}
\usepackage{graphicx}
\usepackage{amscd,amssymb,amsthm}
\usepackage{color}
\usepackage{hyperref}
\usepackage{array,multirow,makecell}

\setcellgapes{1pt}
\makegapedcells
\newcolumntype{R}[1]{>{\raggedleft\arraybackslash }b{#1}}
\newcolumntype{L}[1]{>{\raggedright\arraybackslash }b{#1}}
\newcolumntype{C}[1]{>{\centering\arraybackslash }b{#1}}
\newcounter{minutes}
\divide\time by 60
\newcounter{hours}
\multiply\time by 60 \addtocounter{minutes}{-\time}

\newtheorem{theorem}{Theorem}
\newtheorem{lemma}{Lemma}

\newtheorem{corollary}{Corollary}

\newtheorem{example}{Example}

\title[New integral representations for the Fox--Wright functions]{New integral representations for the Fox-Wright functions and its applications II}

\author[K. Mehrez]{Khaled Mehrez}
\address{Khaled Mehrez. D\'epartement de Math\'ematiques, Facult\'e de Sciences de Tunis, Universit\'e Tunis El Manar, Tunisia.}
\address{D\'epartement de Math\'ematiques ISSAT Kasserine, Universit\'e de Kairouan, Tunisia}
\email{k.mehrez@yahoo.fr}



\begin{document}

\def\thefootnote{}
\footnotetext{ \texttt{File:~\jobname .tex,
          printed: \number\year-0\number\month-\number\day,
          \thehours.\ifnum\theminutes<10{0}\fi\theminutes}
} \makeatletter\def\thefootnote{\@arabic\c@footnote}\makeatother

\maketitle
\begin{abstract}
In this paper our aim is to establish new integral representations for the Fox--Wright function ${}_p\Psi_q[^{(\alpha_p,A_p)}_{(\beta_q,B_q)}|z]$ 
when
$$\mu=\sum_{j=1}^q\beta_j-\sum_{k=1}^p\alpha_k+\frac{p-q}{2}=-m,\;\;m\in\mathbb{N}_0.$$
In particular, closed-form integral expressions are derived for the four parameter Wright function under
a special restriction on parameters. Exponential bounding inequalities are derived for a class of the
Fox-Wright function. Moreover, complete monotonicity property is presented for these functions.
\end{abstract}

\section{Introduction}
We use a definition of the Fox-Wright ( generalized hypergeometric) function by its series 
\begin{equation}\label{11}
{}_p\Psi_q\Big[_{(\beta_1,B_1),...,(\beta_q,B_q)}^{(\alpha_1,A_1),...,(\alpha_p,A_p)}\Big|z \Big]={}_p\Psi_q\Big[_{(\beta_q,B_q)}^{(\alpha_p,A_p)}\Big|z \Big]=\sum_{k=0}^\infty\frac{\prod_{i=1}^p\Gamma(\alpha_l+kA_l)}{\prod_{j=1}^q\Gamma(\beta_l+kB_l)}\frac{z^k}{k!},
\end{equation}
$$\left(\alpha_i,\beta_j\in\mathbb{C},\;\textrm{and}\;\;A_i,B_j\in\mathbb{R}^+\; (i=1,...,p,j=1,...,q)\right),$$
where, as usual,
$$\mathbb{N}=\left\{1,2,3,...\right\},\;\mathbb{N}_0=\mathbb{N}\cup\left\{0\right\},$$
$\mathbb{R},\;\mathbb{R}_+$ and $\mathbb{C}$ stand for the sets of real, positive real and complex numbers, respectively. 
This function was first introduced by Wright \cite{WR} in 1935,  who also derived some of its important
properties including asymptotic behavior.

The convergence conditions and convergence radius of the series at the right-
hand side of (\ref{11}) immediately follow from the known asymptotic of the Euler
Gamma-function. To formulate the results, let us first
introduce the following notations:
\begin{equation}
\Delta=\sum_{j=1}^q B_j-\sum_{i=1}^p A_i,\;\rho=\left(\prod_{i=1}^p A_i^{-A_i}\right)\left(\prod_{j=1}^q B_j^{B_j}\right),\;\;\mu=\sum_{j=1}^q\beta_j-\sum_{k=1}^p\alpha_k+\frac{p-q}{2}
\end{equation}
The defining series in (\ref{11}) converges in the whole complex $z$-plane if $\Delta>-1.$
If $\Delta=-1,$ then the series in (\ref{11}) converges for $|z|<\rho,$ and $|z|=\rho$ under the condition $\Re(\mu)>\frac{1}{2},$ see \cite{24} for details. If, in the definition (\ref{11}), we set
$$A_1=...=A_p=1\;\;\;\textrm{and}\;\;\;B_1=...=B_q=1,$$
we get the relatively more familiar generalized hypergeometric function ${}_pF_q[.]$ given by
\begin{equation}\label{rrrrr}
{}_p F_q\left[^{\alpha_1,...,\alpha_p}_{\beta_1,...,\beta_q}\Big|z\right]=\frac{\prod_{j=1}^q\Gamma(\beta_j)}{\prod_{i=1}^p\Gamma(\alpha_i)}{}_p\Psi_q\Big[_{(\beta_1,1),...,(\beta_q,1)}^{(\alpha_1,1),...,(\alpha_p,1)}\Big|z \Big],\;\;(\alpha_j>0,\;\beta_j\notin\mathbb{Z}_0^-).
\end{equation}
Moreover, both the Wright function $W_{\alpha,\beta}(.)$ and the Mittag-Leffler function $E_{\alpha,\beta}(z)$,
  are particular cases of the Fox-Wright function (\ref{11}):
$$W_{\alpha,\beta}(z)={}_0\Psi_1\left[_{(\beta,\alpha)}^{\;\;-}\Big|z\right],\;E_{\alpha,\beta}(z)={}_1\Psi_1\left[_{(\beta,\alpha)}^{(1,1)}\Big|z\right]$$

Note  important properties for this functions including its Tur\'an, Lazarevi\'c and Wilker type inequalities, was proved by Mehrez \cite{Me1} and Mehrez et al in \cite{SiMe1},\cite{ME3}.

In a recent papers \cite{45},\cite{46},\cite{47}, the author have studied certain  advanced properties of the Fox-Wright function including its new integral representations, the Laplace and Stieltjes transforms, Luke
inequalities, Tur\'an type inequalities and completely monotonicity property are derived. In particular, it was
shown there that the following Fox-Wright functions are completely monotone:
$${}_p\Psi_q\Big[_{(\beta_q,A)}^{(\alpha_p,A)}\Big|-z \Big],\;\;z>0,$$
$${}_{p+1}\Psi_q\Big[_{(\beta_q,1)}^{(\lambda,1),(\alpha_p,A_p)}\Big|\frac{1}{z} \Big],\;\;z>0,$$
and has proved that the Fox's H-function $H_{q,p}^{p,0}[.]$ constitutes the representing measure for the Fox-Wright function ${}_p\Psi_q[.]$, if $\mu>0,$ i.e., \cite[Theorem 1]{45}
\begin{equation}\label{motivation}
{}_p\Psi_q\Big[_{(\beta_q, B_q)}^{(\alpha_p, A_p)}\Big|z\Big]=\int_0^\rho e^{zt}H_{q,p}^{p,0}\left(t\Big|^{(B_q,\beta_q)}_{(A_p,\alpha_p)}\right)\frac{dt}{t}.
\end{equation}
when $\mu>0.$ Here, and in what follows, we use $H_{q,p}^{p,0}[.]$ to denote the Fox's $H$-function, defined by
\begin{equation}\label{5555}
H_{q,p}^{p,0}\left(z\Big|^{(B_q,\beta_q)}_{(A_p,\alpha_p)}\right)=\frac{1}{2i\pi}\int_{\mathcal{L}}\frac{\prod_{j=1}^p\Gamma (A_j s+\alpha_j)}{\prod_{k=1}^q\Gamma (B_k s+\beta_k)}z^{-s}ds,
\end{equation}
where $A_j, B_k>0$ and $\alpha_j,\beta_k>0$.
The contour $\mathcal{L}$ has one of the following forms :
	\begin{itemize}
		\item $\mathcal{L}=\mathcal{L}_{-\infty}$ is a left loop in a horizontal strip starting at the point $-\infty+i\varphi_1$ and terminating at the point $-\infty+i\varphi_2$  with $-\infty< \varphi_1 < \varphi_2 < \infty;$
		\item $\mathcal{L}=\mathcal{L}_{\infty}$ is a right loop in a horizontal strip starting at the point $\infty+i\varphi_1$ and terminating at the point $\infty+i\varphi_2$  with $\infty< \varphi_1 < \varphi_2 < \infty;$
		\item $\mathcal{L}=\mathcal{L}_{i\gamma\infty}$ is a contour starting at the point $\gamma-i\infty$ and terminating at the point $\gamma+i\infty+i\varphi_2,$  where $\gamma\in\mathbb{R}.$ 
		\end{itemize}
		Details regarding the contour and conditions for convergence of the integral in (\ref{5555}) can be
found in \cite[Sections 1.1,1.2]{FF} and \cite{AA}.

 In the course of our investigation, the first main tools is extended some results proved in \cite{45} in the case when $\mu=-m,\;m\in\mathbb{N}_0.$ Secondly, we establish the monotonicity of ratios involving the Fox-Wright functions.

\section{Main results}

The following theorem leads to an extension of the integral equation for the 
$H$-function obtained in (\ref{motivation}) to the case $\mu=-m,\;m\in\mathbb{N}_0.$

\begin{theorem}\label{T1} Suppose that $\mu=-m,\;m\in\mathbb{N}_0$ and $\displaystyle{\sum_{i=1}^p  A_i=\sum_{j=1}^q  B_j}$. If $\gamma \geq1$, then the Fox-Wright function ${}_p\Psi_q[.]$ possesses the following integral representation
\begin{equation}\label{Youssef}
{}_p\Psi_q\Big[_{(\beta_q, B_q)}^{(\alpha_p, A_p)}\Big|z\Big]=\int_0^\rho e^{zt}H_{q,p}^{p,0}\left(t\Big|^{(B_q,\beta_q)}_{(A_p,\alpha_p)}\right)\frac{dt}{t}+\eta\sum_{k=0}^\infty \sum_{j=0}^m \frac{l_{m-j}k^j\rho^k z^k}{k!},
\end{equation}
where the coefficients $\eta$ and $\gamma$ are defined by
\begin{equation}\label{nu}
\eta=(2\pi)^{\frac{p-q}{2}}\prod_{i=1}^p A_i^{\alpha_i-\frac{1}{2}}\prod_{j=1}^q B_j^{\frac{1}{2}-\beta_j},\;\;\gamma=-\min_{1\leq j\leq p}(\alpha_j/A_j),
\end{equation}
and the coefficients $l_r$ satisfy the recurrence relation:
\begin{equation}\label{lr}
l_r=\frac{1}{r}\sum_{n=1}^r q_n l_{r-n},\;\;\;l_0=1,\;\;\textrm{with}\;\;q_n=\frac{(-1)^{n+1}}{n+1}\left[\sum_{i=1}^p\frac{\mathcal{B}_{n+1}(\alpha_i)}{A_i^n}-\sum_{j=1}^q\frac{\mathcal{B}_{n+1}(\beta_j)}{B_j^n}\right],
\end{equation}
where $\mathcal{B}_{n}$ is the Bernoulli polynomial defined via generating function \cite[p. 588]{QQ}
$$\frac{te^{at}}{e^t-1}=\sum_{n=0}^\infty\mathcal{B}_{n}(a)\frac{t^n}{n!},\;\;|t|<2\pi.$$
\end{theorem}
\begin{proof}In \cite[Theorem 2]{Karp3}, the authors found the Mellin transform of the $H$-function when $\mu=-m,\;m\in\mathbb{N}_0,$ that is
\begin{equation}\label{eA}
\frac{\prod_{i=1}^p\Gamma(A_i k+\alpha_i)}{\prod_{j=1}^q \Gamma(B_j k+ \beta_j)}=\int_0^\rho H_{q,p}^{p,0}\left(t\Big|^{(B_q,\beta_q)}_{(A_p,\alpha_p)}\right)t^{k-1}dt+\eta\rho^k\sum_{j=0}^m l_{m-j} k^j,\;\Re(k)>\gamma.
\end{equation}
This implies that
\begin{equation*}
\begin{split}
{}_p\Psi_q\Big[_{(\beta_q, B_q)}^{(\alpha_p, A_p)}\Big|z\Big]&=\sum_{k=0}^\infty\frac{\prod_{i=1}^p\Gamma(A_i k+\alpha_i)z^k}{k!\prod_{j=1}^q \Gamma(B_j k+ \beta_j)}\\
&=\sum_{k=0}^\infty \int_0^\rho H_{q,p}^{p,0}\left(t\Big|^{(B_q,\beta_q)}_{(A_p,\alpha_p)}\right)\frac{(zt)^{k}}{k!}\frac{dt}{t}+\sum_{k=0}^\infty\left(\eta\rho^k\sum_{j=0}^m l_{m-j} \frac{k^j z^k}{k!}\right)\\
&=\sum_{k=0}^\infty \int_0^\rho H_{q,p}^{p,0}\left(t\Big|^{(B_q,\beta_q)}_{(A_p,\alpha_p)}\right)\frac{(zt)^{k}}{k!}\frac{dt}{t}+\eta\sum_{k=0}^\infty \sum_{j=0}^m \frac{l_{m-j}k^j\rho^k z^k}{k!}\\
&=\int_0^\rho H_{q,p}^{p,0}\left(t\Big|^{(B_q,\beta_q)}_{(A_p,\alpha_p)}\right)\left(\sum_{k=0}^\infty \frac{(zt)^{k}}{k!}\right)\frac{dt}{t}+\eta\sum_{k=0}^\infty \sum_{j=0}^m \frac{l_{m-j}k^j\rho^k z^k}{k!}\\
&=\int_0^\rho e^{zt}H_{q,p}^{p,0}\left(t\Big|^{(B_q,\beta_q)}_{(A_p,\alpha_p)}\right)\frac{dt}{t}+\eta\sum_{k=0}^\infty \sum_{j=0}^m \frac{l_{m-j}k^j\rho^k z^k}{k!}.
\end{split}
\end{equation*}
For the exchange of the summation and integration, we use the asymptotic behavior of the $H$-function as $z\rightarrow0$ \cite[Theorem 1.2, Eq. 1.94]{AA}
\begin{equation}\label{asy}
H_{q,p}^{m,n}(z)=\mathcal{O}(z^{-\gamma}), \;|z|\longrightarrow0,
\end{equation}
and the asymptotic behavior of the $H$-function as $z\rightarrow\rho$ \cite[Theorem 1]{Karp3}, we obtain
$$\int_0^\rho t^{k-1} \left|H_{q,p}^{p,0}\left(t\Big|^{(B_q,\beta_q)}_{(A_p,\alpha_p)}\right)\right|dt\leq\int_0^\rho t^{-1} \left|H_{q,p}^{p,0}\left(t\Big|^{(B_q,\beta_q)}_{(A_p,\alpha_p)}\right)\right|dt<M<\infty.$$
Then, we are in position to apply the Lebesgue dominated convergence theorem. This completes the proof of Theorem \ref{T1}.
\end{proof}

Recall that a function $f : ( 0 ,\infty) \longrightarrow (0,\infty)$ is called completely monotonic if
$(-1)^n f^{(n)}(x)\geq0$ for $x > 0$ and $n\in\mathbb{N}_0.$ The celebrated Bernstein theorem asserts
that completely monotonic functions are precisely those that can be expressed by the
Laplace transform of a non-negative measure.

\begin{corollary}\label{C1} Suppose that $\mu=0,\;\gamma\geq1$ and $\displaystyle{\sum_{i=1}^p  A_i=\sum_{j=1}^q  B_j}.$ Then, the Fox-Wright function ${}_p\Psi_q[.]$ possesses the following integral representation
\begin{equation}\label{E1}
{}_p\Psi_q\Big[_{(\beta_q, B_q)}^{(\alpha_p, A_p)}\Big|-z\Big]-\eta e^{-\rho z}=\int_0^\rho e^{-zt}H_{q,p}^{p,0}\left(t\Big|^{(B_q,\beta_q)}_{(A_p,\alpha_p)}\right)\frac{dt}{t},\;\;z\in\mathbb{R}.
\end{equation}
Moreover, if the function $H_{q,p}^{p,0}[.]$ is non-negative, then the function 
$$z\mapsto {}_p\Psi_p\Big[_{(\beta_p, A)}^{(\alpha_p, A)}\Big|-z\Big]-\eta e^{-\rho z},$$
is completely monotonic on $(0,\infty).$ 
\end{corollary}
\begin{proof}The application of Theorem \ref{T1} for $m=0$ yields 
\begin{equation}
{}_p\Psi_q\Big[_{(\beta_q, B_q)}^{(\alpha_p, A_p)}\Big|z\Big]=\int_0^\rho e^{zt}H_{q,p}^{p,0}\left(t\Big|^{(B_q,\beta_q)}_{(A_p,\alpha_p)}\right)\frac{dt}{t}+\eta e^{\rho z},
\end{equation}
which implies that (\ref{E1}) holds. Now, suppose  that the $H$-function $H_{p,p}^{p,0}[.]$, then by means of (\ref{E1}), we deduce that all prerequisites of the Bernstein Characterization
Theorem for the complete monotone functions are fulfilled.
\end{proof}
\begin{example}The four parameters Wright function is defined by the series 
\begin{equation}\label{K28}
\phi\left((\mu_1,a),(\nu_1,b);z\right)=\sum_{k=0}^\infty\frac{z^k}{\Gamma(a+k\mu_1)\Gamma(b+k\nu_1)},\;\mu_1,\nu_1\in\mathbb{R},\;a,b\in\mathbb{C}.
\end{equation}
The series on the right-hand side of (\ref{K28}) is absolutely convergent for all $z\in\mathbb{C}$ if $\mu_1+\nu_1>0.$ If $\mu_1+\nu_1=0,$ the series is absolutely convergent for $|z|<|\mu_1|^{\mu_1} |\nu_1|^{\nu_1}$ and $|z|=|\mu_1|^{\mu_1} |\nu_1|^{\nu_1}$ under the condition $\Re(a+b)>2.$  Some of the basic properties of the four parameters Wright function was proved in \cite{ER59}. So, by means of formula (\ref{E1}) we deduce that the four parameters Wright function $\phi\left((\mu_1,a),(\nu_1,b);z\right)$ admits the following integral representation:
\begin{equation}\label{280}
\phi\left((\mu_1,a),(\nu_1,b);z\right)=\int_0^{{\mu_1}^{\mu_1}\nu_1^{\nu_1}}e^{zt}H_{2,1}^{1,0}\left[t\Big|_{(1,1)}^{(\mu_1,a),(\nu_1,b)}\right]\frac{dt}{t}+\frac{{\mu_1}^{\frac{1}{2}-a}{\nu_1}^{\frac{1}{2}-b}}{\sqrt{2\pi}}e^{{\mu_1}^{\mu_1}{\nu_1}^{\nu_1} z},
\end{equation}
where $a,b,\mu_1, \nu_1\in\mathbb{R}$ for which $\mu_1+\nu_1=1$ and $a+b=3/2.$
\end{example}

As a consequence, we derive the finite Laplace Transform for the function 
$$t\mapsto t^{-1}H_{2,1}^{1,0}\left[t\Big|_{(1,1)}^{(1/2,1/2),(1/2,1)}\right]$$
 in $(0,1/2)$. Recall that the finite Laplace Transform of a continuous ( or an almost piecewise continuous) function $f(t)$ in $(0,T)$ is denoted by $$\mathcal{L}_T{f(t)}=\bar{f}(s,T)=\int_0^T e^{-st}f(t)dt.$$
We note that $\mathcal{L}_T f$ is actually the Laplace transform of the function $f$ which vanishes outside of the interval $(0,T).$
\begin{example}Letting in (\ref{K28}) and (\ref{280}), the values $\nu_1=\mu_1=a=1/2$ and $b=1$ and using the 
Legendre Duplication Formula $$\Gamma(z)\Gamma(z+1/2)=2^{1-2z}\sqrt{\pi}\Gamma(2z),$$ we get the following curious integral evaluation:
\begin{equation}
\frac{e^{-2z}-e^{-z}}{\sqrt{\pi}}=\int_0^{\frac{1}{2}} e^{-zt} H_{2,1}^{1,0}\left[t\Big|_{(1,1)}^{(1/2,1/2),(1/2,1)}\right]\frac{dt}{t},\;z\in\mathbb{R}.
\end{equation}
\end{example}

In the next result we show that the function ${}_{p+1}\Psi_q[.]-\eta{}_1F_0[.]$ is a  Stieltjes transform.

\begin{corollary} Let $\sigma>0$ and $z\in\mathbb{C}$, such that $|\arg(1+z)|<\pi$ and $|z|<1.$ In addition, assume that the hypotheses of Corollary \ref{C1} are satisfied. If $H_{q,p}^{p,0}[.]$ is non-negative, then, the following representation holds true:
\begin{equation}\label{eaaa}
f(z):={}_{p+1}\Psi_q\Big[_{\;\;\;(\beta_q, B_q)}^{(\sigma,1),(\alpha_p, A_p)}\Big|-z\Big]-\eta\;{}_1F_0(\sigma;-;-\rho z)=\Gamma(\sigma)\int_0^\rho H_{q,p}^{p,0}\left(t\Big|^{(B_q,\beta_q)}_{(A_p,\alpha_p)}\right)\frac{dt}{t(1+tz)^\sigma}.
\end{equation}
In particular, $f(z)$ is completely monotonic on $(0,\infty).$
\end{corollary}
\begin{proof} Employing the generalized binomial expansion
$$(1+z)^{-\sigma}=\sum_{k=0}^\infty (\sigma)_k\frac{(-1)^k z^k}{k!},\;z\in\mathbb{C},\;\;|z|<1,$$
with the formula (\ref{eA}) and the right hand side of (\ref{eaaa}) we obtain 
\begin{equation}
\begin{split}
\Gamma(\sigma)\int_0^\rho H_{q,p}^{p,0}\left(t\Big|^{(B_q,\beta_q)}_{(A_p,\alpha_p)}\right) \frac{dt}{t(1+tz)^\sigma}&=\Gamma(\sigma)\sum_{k=0}^\infty (\sigma)_k\frac{(-1)^k z^k}{k!} \left[\int_0^\rho t^{k-1}H_{q,p}^{p,0}\left(t\Big|^{(B_q,\beta_q)}_{(A_p,\alpha_p)}\right) dt\right]\\
&=\Gamma(\sigma)\sum_{k=0}^\infty (\sigma)_k\frac{(-1)^k z^k}{k!} \left[\frac{\prod_{i=1}^p\Gamma(A_i k+\alpha_i)}{\prod_{j=1}^q \Gamma(B_j k+ \beta_j)}-\eta \rho^k \right]\\
&=\sum_{k=0}^\infty \frac{\Gamma(\sigma+k)\prod_{i=1}^p\Gamma(\alpha_i+k A_i)}{\prod_{j=1}^q\Gamma(\beta_j+kB_j)}\frac{ (-z^k}{k!}-\eta\sum_{k=0}^\infty\frac{\Gamma(\sigma+k)(-\rho z)^k}{k!} \\
&={}_{p+1}\Psi_q\Big[_{\;\;\;(\beta_q, B_q)}^{(\sigma,1),(\alpha_p, A_p)}\Big|-z\Big]-\eta\;{}_1F_0(\sigma;-;-\rho z).
\end{split}
\end{equation}
This completes the proof.
\end{proof}

\begin{corollary}Suppose that $\mu=-1$ and $\displaystyle{\sum_{i=1}^p  A_i=\sum_{j=1}^q  B_j}.$ Then the following integral representation
\begin{equation}\label{E1E}
g(z):={}_p\Psi_q\Big[_{(\beta_q, B_q)}^{(\alpha_p, A_p)}\Big|-z\Big]-\eta (l_1-\rho z) e^{-\rho z}=\int_0^\rho e^{-zt}H_{q,p}^{p,0}\left(t\Big|^{(B_q,\beta_q)}_{(A_p,\alpha_p)}\right) dt,\;z\in\mathbb{R}
\end{equation}
holds true. Moreover, if $H_{q,p}^{p,0}[.]$ is non-negative, then the function $g(z)$ is completely monotonic on $(0,\infty).$
\end{corollary}
\begin{proof}From Theorem \ref{T1}, when $\mu=-1,$ we have
\begin{equation}\label{MM}
{}_p\Psi_q\Big[_{(\beta_q, B_q)}^{(\alpha_p, A_p)}\Big|-z\Big]=\int_0^\rho e^{-zt}H_{q,p}^{p,0}\left(t\Big|^{(B_q,\beta_q)}_{(A_p,\alpha_p)}\right)\frac{dt}{t}+\eta (l_1-\rho z) e^{- \rho z}.
\end{equation}
\end{proof}

\begin{example} The four parameters Wright function  $\phi\left((\mu,a),(\nu,b);z\right)$ possesses the following integral representation 
\begin{equation}\label{28}
\phi\left((\mu,a),(\nu,b);z\right)=\int_0^{\mu^\mu\nu^\nu}e^{zt}H_{2,1}^{1,0}\left[t\Big|_{(1,1)}^{(\mu,a),(\nu,b)}\right]\frac{dt}{t}+\frac{\mu^{\frac{1}{2}-a}\nu^{\frac{1}{2}-b}}{\sqrt{2\pi}}\left(l_1+\mu^\mu\nu^\nu z\right)e^{\mu^\mu\nu^\nu z},
\end{equation}
where 
$$l_1=\frac{1}{12}-\frac{6a^2-6a+1}{12\mu}-\frac{6b^2-6b+1}{12\nu},$$
and $a,b,\mu, \nu$ be a real number such that $\mu+\nu=1$ and $a+b=1/2.$
\end{example}

The following lemma is called the Jensen's integral inequality, for more details, one may see \cite[ Chap. I, Eq. (7.15)]{DM}.

\begin{lemma}\label{l0} Let $\omega$ be a non-negative measure and let $\varphi\geq0$ be a convex function. Then for all $f$ be a integrable function we have
\begin{equation}\label{jensen}
\varphi\left(\int f d\nu\Big/\int d\nu\right)\leq \int \varphi\circ f d\nu\Big/\int d\nu.
\end{equation}
\end{lemma}

In the next theorem we present a new Luke type inequality when  $\mu=0.$

\begin{theorem} Keep the notations and constraints of hypotheses of Corollary \ref{C1}. Assume that the function $H_{q,p}^{p,0}[.]$ is non-negative, then  the following two-sided bounding inequality holds true:
\begin{equation}\label{KM1}
\begin{split}
 \Psi_0\;e^{-(\Psi_1/\Psi_0) z}+\eta\;e^{-\rho z}\leq {}_p\Psi_q\Big[_{(\beta_q, B_q)}^{(\alpha_p, A_p)}\Big|-z\Big]&\leq\left(\Psi_0-\frac{\Psi_1}{\rho}\right)+\left(\eta+\frac{\Psi_1}{\rho}\right)e^{-\rho z},
\end{split}
\end{equation}
where
\begin{equation*}
\begin{split}
\Psi_0:&=\frac{\prod_{i=1}^p\Gamma(\alpha_i)}{\prod_{j=1}^q\Gamma(\beta_j)}-\eta,\\
\Psi_1:&=\frac{\prod_{i=1}^p\Gamma(\alpha_i+A_i)}{\prod_{j=1}^q\Gamma(\beta_j+B_j)}-\eta\rho.
\end{split}
\end{equation*}
\end{theorem}
\begin{proof} Letting $\varphi_{z}(t)= e^{-zt},\;f(t)=t,$ and 
$$d\nu(t)= H_{p,p}^{p,0}\left(t\Big|^{(A,\beta_p)}_{(A,\alpha_p)}\right)\frac{dt}{t}.$$
From (\ref{eA}) we get
$$\int_0^\rho d\nu(t)=\frac{\prod_{i=1}^p\Gamma(\alpha_i)}{\prod_{j=1}^q\Gamma(\beta_j)}-\eta=\Psi_0,\;\textrm{and}\;\;\int_0^\rho f(t)d\nu(t)=\frac{\prod_{i=1}^p\Gamma(\alpha_i+A_i)}{\prod_{j=1}^p\Gamma(\beta_j+B_j)}-\eta\rho=\Psi_1,$$
and using (\ref{Youssef}) when $m=0$ we find 
$$\int_0^\rho \phi_z(f(t))d\nu(t)={}_p\Psi_p\Big[_{(\beta_q,A)}^{(\alpha_p,A)}\Big|-z\Big]-\eta e^{-\rho z}.$$
Hence, Lemma \ref{l0} completes the proof of the lower bound of inequalities (\ref{KM1}). In order to demonstrate the upper bound, we will apply the converse Jensen inequality, due to Lah and Ribari\'c, which reads as follows. Set
$$A(f)=\int_m^M f(s)d\sigma(s)\Big/ \int_m^M d\sigma(s),$$
where $\sigma$ is a non-negative measure and $f$ is a continuous function. If $-\infty<m<M<\infty$ and $\varphi$ is convex on $[m,M],$ then according to \cite[Theorem 3.37]{PE}
\begin{equation}\label{CC}
(M-m)A(\varphi(f))\leq (M-A(f))\varphi(m)+(A(f)-m)\varphi(M).
\end{equation}
Setting $\varphi_z(t)=e^{-zt},\;d\sigma(t)=d\nu(t),\;f(s)=s$ and $[m,M]=[0,\rho]$, we complete the proof of the upper bound in (\ref{KM1}).
\end{proof}

In view of inequalities (\ref{KM1}) and the Laplace transform of the function $x^{\lambda-1}{}_p\Psi_q[x]$ \cite[Eq. (7)]{P}
\begin{equation}
\int_0^\infty e^{-t}t^{\lambda-1}{}_p\Psi_q\left[_{(\beta_q,B_q)}^{(\alpha_p,A_p)}\Big|zt\right]dt={}_{p+1}\Psi_q\left[_{(\beta_q, B_q)}^{(\lambda,1),(\alpha_p, A_p)}\Big|z\right]
\end{equation}
and make use of the following known formula
$$\int_0^\infty t^\lambda e^{-\sigma t}dt=\frac{\Gamma(\lambda+1)}{\sigma^{\lambda+1}},\;(\lambda>-1,\;\sigma>0),$$
we can deduce the new following inequalities for the function ${}_{p+1}\Psi_q[.]:$

\begin{corollary} Let $\lambda>0$ and suppose the hypotheses of Corollary \ref{C1} are satisfied. If the function $H_{q,p}^{p,0}[.]$ is non-negative, then  the following two--sided bounding inequality holds true:
\begin{equation}\label{9++}
\begin{split}
\frac{\eta\Gamma(\lambda)}{(1+\rho z)^\lambda}+\frac{\Psi_0\Gamma(\lambda)}{(1+(\Psi_1/\Psi_0) z)^\lambda}\leq {}_{p+1}\Psi_q\left[_{(\beta_q, B_q)}^{(\lambda,1),(\alpha_p, A_p)}\Big|-z\right]\leq\left(\Gamma(\lambda)\Psi_0-\frac{\Gamma(\lambda)\Psi_1}{\rho}\right)+\frac{\Gamma(\lambda)\left(\eta+\frac{\Psi_1}{\rho}\right)}{(1+\rho z)^\lambda}.
\end{split}
\end{equation}
\end{corollary}

\begin{theorem} Keep the notations and constraints of hypotheses of Corollary \ref{C1}. Assume that the function $H_{q,p}^{p,0}[.]$ is non-negative, then the following inequality
\begin{equation}\label{99999}
\frac{\Gamma(\sigma)\Psi_0}{(1+(\Psi_1/\psi_0)z)^\sigma}+\eta\;{}_1F_0(\sigma;-;-\rho z)\leq{}_{p+1}\Psi_q\Big[_{\;\;\;(\beta_q, B_q)}^{(\sigma,1),(\alpha_p, A_p)}\Big|-z\Big],
\end{equation}
is valid for all $\sigma>0$ and $|z|<1.$
\end{theorem}
\begin{proof}We set $\varphi(u)=u^\sigma,\;\sigma>0,\;f(t)=1/(1+tz)$ and $d\nu(t)=t^{-1}\Gamma(\sigma) H_{q,p}^{p,0}\left(t\Big|^{(B_q,\beta_q)}_{(A_p,\alpha_p)}\right) dt.$ By (\ref{eA}) we have
\begin{equation}\label{4}
\begin{split}
\int_0^\rho d\nu(t)&=\Gamma(\sigma) \int_0^\rho H_{q,p}^{p,0}\left(t\Big|^{(B_q,\beta_q)}_{(A_p,\alpha_p)}\right)\frac{dt}{t}\\
&=\frac{\Gamma(\sigma)\prod_{i=1}^p\Gamma(\alpha_i)}{\prod_{j=1}^q\Gamma(\beta_j)}-\eta\Gamma(\sigma).
\end{split}
\end{equation}
Moreover (\ref{eaaa}), reads 
\begin{equation}\label{44}
\int_0^\rho f(t)d\nu(t)={}_{p+1}\Psi_q\Big[_{\;\;\;(\beta_q, B_q)}^{(1,1),(\alpha_p, A_p)}\Big|-z\Big]-\frac{\eta}{1+\rho z},
\end{equation}
and 
\begin{equation}\label{444}
\int_0^\rho \varphi(f(t))d\nu(t)={}_{p+1}\Psi_q\Big[_{\;\;\;(\beta_q, B_q)}^{(\sigma,1),(\alpha_p, A_p)}\Big|-z\Big]-\eta\;{}_1F_0(\sigma;-;-\rho z).
\end{equation}
By means of Lemma \ref{l0} we obtain
\begin{equation}\label{444++}
\Gamma(\sigma)\Psi_0^{1-\sigma}\left({}_{p+1}\Psi_q\Big[_{\;\;\;(\beta_q, B_q)}^{(1,1),(\alpha_p, A_p)}\Big|-z\Big]-\frac{\eta}{1+\rho z}\right)^\sigma\leq {}_{p+1}\Psi_q\Big[_{\;\;\;(\beta_q, B_q)}^{(\sigma,1),(\alpha_p, A_p)}\Big|-z\Big]-\eta\;{}_1F_0(\sigma;-;-\rho z)
\end{equation}
By virtue of the left-hand side of inequality (\ref{9++}) and (\ref{444++}) we conclude the inequality (\ref{99999}).
\end{proof}

The next lemma is in fact the so-called the Chebyshev integral inequality, see \cite[p. 40]{DM}.

\begin{lemma}\label{l1} If $f,g:[a,b]\longrightarrow\mathbb{R}$ are synchoronous (both increasing  or decreasing) integrable functions, and $p:[a,b]\longrightarrow\mathbb{R}$  is a positive integrable function, then 
\begin{equation}\label{OO}
\int_a^b p(t)f(t)dt\int_a^b p(t)g(t)dt\leq \int_a^b p(t)dt\int_a^b p(t)f(t)g(t)dt.
\end{equation}
Note that if $f$ and $g$ are asynchronous (one is decreasing and the other is increasing),
then (\ref{OO}) is reversed. 
 \end{lemma}
\begin{theorem}Assume that the hypotheses of Corollary \ref{C1} are satisfied. Suppose that $\delta,\sigma>0.$ If $H_{q,p}^{p,0}[.]$ is non-negative, then the function 
\begin{equation}
F:=F\left[^{(\sigma,1),(\alpha_p,A_p)}_{\;\;(\beta_q,B_q)}\Big|\delta;z\right]=\frac{{}_{p+1}\Psi_q\Big[_{\;\;\;(\beta_q+\delta B_q, B_q)}^{(\sigma,1),(\alpha_p+\delta A_p, A_p)}\Big|-z\Big]-\eta\;{}_1F_0(\sigma;-;-\rho z)}{{}_{p+1}\Psi_q\Big[_{\;\;\;(\beta_q, B_q)}^{(\sigma,1),(\alpha_p, A_p)}\Big|-z\Big]-\eta\;{}_1F_0(\sigma;-;-\rho z)},
\end{equation}
is increasing on $(0,1)$. In addition, the function $F(z)$ is decreasing on $(0,1)$ for each $\delta<0$ and $\sigma>0.$
\end{theorem}
\begin{proof}In view of (\ref{eaaa}), using the following property of the Fox $H$-function \cite[Property 1.5, p. 12]{AA}
$$H_{q,p}^{n,m}\left[^{(A_p,\alpha_p+\delta A_p)}_{(B_q,\beta_q+\delta B_q)}\Big| z\right]=z^\delta H_{q,p}^{n,m}\left[^{(A_p,\alpha_p)}_{(B_q,\beta_q)}\Big| z\right],$$
we can rewrite the function $F$ as follows:
\begin{equation*}
\begin{split}
F\left[^{(\sigma,1),(\alpha_p,A_p)}_{\;\;(\beta_q,B_q)}\Big|\delta;z\right]&=\frac{\int_0^\rho H_{q,p}^{p,0}\left[^{(A_p,\alpha_p+\delta A_p)}_{(B_q,\beta_q+\delta B_q)}\Big| t\right]\frac{dt}{t(1+tz)^\sigma}}{\int_0^\rho H_{q,p}^{p,0}\left[^{(A_p,\alpha_p)}_{(B_q,\beta_q)}\Big| t\right]\frac{dt}{t(1+tz)^\sigma}}\\
&=\frac{\int_0^\rho t^{\delta-1}H_{q,p}^{p,0}\left[^{(A_p,\alpha_p)}_{(B_q,\beta_q)}\Big| t\right]\frac{dt}{(1+tz)^\sigma}}{\int_0^\rho H_{q,p}^{p,0}\left[^{(A_p,\alpha_p)}_{(B_q,\beta_q)}\Big| t\right]\frac{dt}{t(1+tz)^\sigma}}.
\end{split}
\end{equation*}
Now, we consider the functions $p,f,g:[0,\rho]\longrightarrow\mathbb{R},$ defined by
$$p(t)=t^{-1}(1+tz)^{-\sigma}H_{q,p}^{p,0}\left[^{(A_p,\alpha_p)}_{(B_q,\beta_q)}\Big| t\right],\;f(t)=t^{\delta},\;g(t)=\frac{t}{1+tz}.$$
Observe that the functions $f$ and $g$ are increasing, thus, by means of Lemma \ref{l1}, we infer
\begin{equation}\label{TYT1}
\left(\int_0^\rho t^{\sigma-1} H_{q,p}^{p,0}\left[^{(A_p,\alpha_p)}_{(B_q,\beta_q)}\Big| t\right]\frac{dt}{(1+tz)^\sigma}\right)\left(\int_0^\rho H_{q,p}^{p,0}\left[^{(A_p,\alpha_p)}_{(B_q,\beta_q)}\Big| t\right]\frac{dt}{(1+tz)^{\sigma+1}}\right)
\end{equation}
$$\;\;\;\;\;\;\;\;\;\;\;\;\;\;\;\;\;\;\leq\left(\int_0^\rho H_{q,p}^{p,0}\left[^{(A_p,\alpha_p)}_{(B_q,\beta_q)}\Big| t\right]\frac{dt}{t(1+tz)^\sigma}\right)\left(\int_0^\rho t^\sigma H_{q,p}^{p,0}\left[^{(A_p,\alpha_p)}_{(B_q,\beta_q)}\Big| t\right]\frac{dt}{(1+tz)^{\sigma+1}}\right).$$
On the other hand, we have
\begin{equation}\label{TYT}
\frac{1}{\sigma}\left[\int_0^\rho H_{q,p}^{p,0}\left[^{(A_p,\alpha_p)}_{(B_q,\beta_q)}\Big| t\right]\frac{dt}{t(1+tz)^\sigma}\right]^2\frac{\partial}{\partial z}F\left[^{(\sigma,1),(\alpha_p,A_p)}_{\;\;(\beta_q,B_q)}\Big|\delta;z\right]=
\end{equation}
\begin{equation*}
\begin{split}
&=
\left(\int_0^\rho H_{q,p}^{p,0}\left[^{(A_p,\alpha_p)}_{(B_q,\beta_q)}\Big| t\right]\frac{dt}{t(1+tz)^\sigma}\right)\left(\int_0^\rho t^\sigma H_{q,p}^{p,0}\left[^{(A_p,\alpha_p)}_{(B_q,\beta_q)}\Big| t\right]\frac{dt}{(1+tz)^{\sigma+1}}\right)\\
&-\left(\int_0^\rho t^{\sigma-1} H_{q,p}^{p,0}\left[^{(A_p,\alpha_p)}_{(B_q,\beta_q)}\Big| t\right]\frac{dt}{(1+tz)^\sigma}\right)\left(\int_0^\rho H_{q,p}^{p,0}\left[^{(A_p,\alpha_p)}_{(B_q,\beta_q)}\Big| t\right]\frac{dt}{(1+tz)^{\sigma+1}}\right).
\end{split}
\end{equation*}
By (\ref{TYT1}) and (\ref{TYT}) we deduce that the function $z\mapsto F(z)$ is increasing on $(0,1)$ for all $\sigma>0$ and $\delta>0.$ Moreover, if $\delta<0$ then the inequality (\ref{TYT1}) is reversed and consequently the function $z\mapsto F(z)$ is decreasing on $(0,1)$ for all $\sigma>0$ and $\delta<0.$ Now, the proof of this theorem is completed.
\end{proof}

\noindent {\bf Acknowledgements:} The author is grateful to the reviewers for the suggestions that help to improve the
paper.\\

\end{document}